\documentclass[12pt, a4paper]{amsart}
\usepackage[left=2.2cm, right=2.2cm, top=3cm, bottom=3.5cm]{geometry}
\usepackage{amsmath,tabu}
\usepackage{amssymb}
\usepackage{amsthm}
\usepackage[utf8]{inputenc} 
\usepackage[T1]{fontenc}
\usepackage{mathrsfs}
\usepackage{caption}
\usepackage{url}
\usepackage{enumitem}
\usepackage[toc,page]{appendix}
\usepackage{tabu}
\usepackage{longtable}

\usepackage{tikz}
\def\DynkinNodeSize{3.5mm}
\def\DynkinArrowLength{3mm}
\usetikzlibrary{arrows,decorations.markings}
\usetikzlibrary{arrows,decorations.markings}
\usetikzlibrary{shapes}
\usetikzlibrary{matrix}
\usetikzlibrary{graphs}
\usetikzlibrary{backgrounds}

\usepackage{tikz-cd}
\usetikzlibrary{arrows}

\tikzset{
commutative diagrams/.cd,
arrow style=tikz,
diagrams={>=latex}}

\begin{document}

\newtheorem*{thmn}{Theorem}
\newtheorem{thmx}{Theorem}
\newtheorem{thmy}{Theorem}
\renewcommand{\thethmy}{\Alph{thmy}}

\newtheorem{thm}{Theorem}[section]
\newtheorem{lem}[thm]{Lemma}
\newtheorem{prop}[thm]{Proposition}
\newtheorem{cor}[thm]{Corollary}
\newtheorem*{cor2}{Corollary}
\newtheorem*{conjecture}{Conjecture}
\theoremstyle{remark}
\newtheorem{remark}[thm]{Remark}
\theoremstyle{definition}
\newtheorem{notation}[thm]{Notation}
\newtheorem{definition}[thm]{Definition}

\renewcommand{\thethmx}{\Alph{thmx}}

\newenvironment{example}[1][Example]{\begin{trivlist}
\item[\hskip \labelsep {\bfseries #1}]}{\end{trivlist}}

\title[Bounds on the number of simple modules]{Bounds on the number of simple modules in blocks of finite groups of Lie type}

\date{\today}

\author{Ruwen Hollenbach}
\address{FB Mathematik, TU Kaiserslautern,
         67663 Kaiserslautern, Germany.}
\email{hollenbach@mathematik.uni-kl.de}



\begin{abstract}
Let $G$ be a simple, simply connected linear algebraic group of exceptional type defined over $\mathbb{F}_q$ with Frobenius endomorphism $F: G \to G$. In this work we give upper bounds on the number of simple modules in the quasi-isolated $\ell$-blocks of $G^F$ and $G^F/Z(G^F)$ when $\ell$ is bad for $G$. 

\end{abstract}

\maketitle

\section{Introduction}
\noindent
Let $G$ be a simple, simply connected linear algebraic group of exceptional type over $\mathbb{F}_q$ with Frobenius endomorphism $F: G \to G$.
Let $\ell$ be a prime not dividing $q$.

When $\ell$ is a good prime for $G$, explicit basic sets for the quasi-isolated $\ell$-blocks of $G^F$ were determined by the author in \cite{Ruwen} using \cite[Theorem A]{Geck2} and $e$-Harish-Chandra theory, as established by Cabanes--Enguehard in \cite{Cabanes-Enguehardtwisted}. In particular, when $\ell$ is good for $G$, we know the number of irreducible Brauer characters in these blocks.

When $\ell$ is bad for $G$, however, the assertion of \cite[Theorem A]{Geck2} no longer holds, and very little is known about the number of irreducible Brauer characters in the quasi-isolated $\ell$-blocks of $G^F$. 
In this paper we give upper bounds for the number of irreducible Brauer characters in the quasi-isolated $\ell$-blocks of $G^F$ when $\ell$ is bad for $G$.  \\


\noindent
For $s \in G^{*F}$, let $l_s$ denote the number of irreducible Brauer characters in $\mathcal{E}_\ell(G^F,s)$ (see Theorem \ref{Lusztigseries blocks}). We prove the following replacement of \cite[Theorem A]{Geck2}.
\begin{thm} \label{Theorem A}
Suppose that $G$ is a simple, simply connected linear algebraic group of exceptional type defined over $\mathbb{F}_q$ with Frobenius endomorphism $F: G \to G$.
Let $\ell \nmid q$ be a bad prime for $G$. If $s \in G^{*F}$ is a semisimple quasi-isolated $\ell'$-element then 
\[
l_s \leq \begin{cases} 3 \text{ }|\mathcal{E}(G^F,s)|, \text{ if } G=E_6, \ell=3 \text{ and } C_{G^*}(s)=A_5 \times A_1, \\
			2 \text{ } |\mathcal{E}(G^F,s)|, \text{ if } G=E_7, \ell=2 \text{ and } C_{G^*}(s)=A_5 \times A_2, \\
			|\mathcal{E}(G^F,s)|+ \sum_{1 \neq t}|\mathcal{E}(G^F,st)|, \text{ otherwise};
\end{cases}
\] where $t$ runs over the $\ell$-elements of $C_{G^*}(s)^F$ such that 
\begin{enumerate}
\item $st$ is quasi-isolated; and
\item $C_{G^*}(st)^F \neq C_{G^*}^\circ(st)^F$ or $C_{G^*}^\circ(st)$ is not of type $A$.
\end{enumerate}
\end{thm}

\noindent
Using Theorem \ref{Theorem A} and $e$-Harish-Chandra theory, as established in \cite{Enguehard} and \cite{Malle-Kessar}, we can determine explicit upper bounds for the number of irreducible Brauer characters in each individual block in $\mathcal{E}_\ell(G^F,s)$ where $s \in G^{*F}$ is a semisimple quasi-isolated $\ell'$-element. We then use these upper bounds to check the Malle--Robinson conjecture for the quasi-isolated blocks of the finite groups of exceptional Lie type. Together with the results in \cite{Ruwen} this yields the following.


\begin{thm} \label{Theorem B}
Let $G$ be a simple, simply connected linear algebraic group of exceptional type defined over $\mathbb{F}_q$ with Frobenius endomorphism $F: G \to G$. Let $\ell \nmid q$ be a bad prime for $G$ and let $B$ be a quasi-isolated $\ell$-block of $G^F$ or $G^F/Z(G^F)$. Then the Malle--Robinson conjecture holds for $B$ unless, possibly, when $G^F=E_8(q)$ and $B$ is the block numbered 3 or 8 in the table in \cite[page 358]{Enguehard}.


\end{thm}
\noindent
Using the reduction of Bonnaf\'e--Dat--Rouquier \cite[Theorem 7.7]{BDR}, we can then prove the following corollary to Theorem \ref{Theorem B}.

\begin{cor}\label{Theorem C}
Let $H$ be a finite quasi-simple group of exceptional Lie type. Let $\ell$ be a prime and let $B$ be an $\ell$-block of $H$. Then $(H,B)$ is not a minimal counterexample unless, possibly,  when $H=E_8(q)$ and $B$ is numbered $3$ or $8$ in the table in \cite[page 358]{Enguehard}.
\
\end{cor}

\section{Main tools}
\noindent
Let $G$ be a connected reductive group defined over $\mathbb{F}_q$ with Frobenius endomorphism $F$. In this section we will introduce the notation and the main tools used in the proofs of Theorems \ref{Theorem A} and \ref{Theorem B}. \\

\noindent
Let $R_{L \subseteq P}^G$ denote Lusztig induction from an $F$-stable Levi subgroup $L$ contained in a parabolic subgroup $P \subseteq G$ to $G$ and let $^*\!R_{L \subseteq P}^G$ denote the adjoint functor (see \cite[11.1 Definition]{Digne-Michel}). By \cite{Bonnafe-Michel}, we can (and will) omit the parabolic subgroup from the subscript for most $G$ and $q$. In the few cases where we still do not know if the parabolic subgroup can be omitted, the proofs of Theorems \ref{Theorem A} and \ref{Theorem B} are immediate, which is why we omit the parabolic anyway.

Let $G^*$ be a group in duality with $G$ with respect to an $F$-stable maximal torus $T$ of $G$ (see \cite[Definition 13.10]{Digne-Michel}). By results of Lusztig, $\operatorname{Irr}(G^F)$ is a disjoint union of so-called (rational) Lusztig series $\mathcal{E}(G^F,s)$, where $s$ runs over the $G^F$-conjugacy classes of semisimple elements of the dual group $G^*$ (see \cite[Definition 8.23]{Cabanes-Enguehard}). 
Recall the following classical result about the block theory of finite groups of Lie type.

\begin{thm}[{\cite[Théorème 2.2]{Broue}, \cite[Theorem 3.1]{Hiss}}]\label{Lusztigseries blocks}
Let $s \in G^{*F}$ be a semisimple $\ell'$-element. Then we have the following.\\
\begin{tabular}{rp{13,7cm}}
(a) & The set $
\mathcal{E}_\ell(G^F,s):= \bigcup_{t \in C_{G^*}(s)_\ell^F} \mathcal{E}(G^F,st)$
 is a union of $\ell$-blocks of $G^F$. \\
(b) & Any $\ell$-block contained in $\mathcal{E}_\ell(G^F,s)$ contains a character of $\mathcal{E}(G^F,s)$.
\end{tabular}
\end{thm}

\noindent
For good primes $\ell$ we have the following stronger result. Let $\chi$ be an ordinary irreducible character of $G^F$. We denote the restriction of $\chi$ to the $\ell$-regular elements of $G^F$ by $\chi^\circ$. We define $\hat{\mathcal{E}}(G^F,s):=\{\chi^\circ \mid \chi \in \mathcal{E}(G^F,s)\}$.

\begin{thm}[{\cite[Theorem A]{Geck2}}] \label{Geck}
Let $G$ be a connected reductive group defined over $\mathbb{F}_q$ with Frobenius endomorphism $F:G \to G$. Assume that $\ell$ is a good prime for $G$ not dividing the order of $(Z(G)/Z^\circ(G))_F$. Let $s \in G^{*F}$ be an $\ell'$-element. Then $\mathcal{E}(G^F,s)$ is an ordinary basic set for the union of blocks $\mathcal{E}_\ell(G^F,s)$.
\end{thm}

\noindent
Note that the assertion of Theorem \ref{Geck} no longer holds for bad primes (see section 1.2 of \cite{Geck-Hiss} for a counterexample). The crux of this work is therefore to find a replacement for Theorem \ref{Geck} when $\ell$ is a bad prime.  \\

\noindent
The following results are the key tools in the proof of Theorem \ref{Theorem A}.

\begin{lem} \label{t im Zentrum}
Let $G$ be a connected reductive group defined over $\mathbb{F}_q$ with Frobenius endomorphism $F:G \to G$. Let $s \in G^*$ be a semisimple $\ell'$-element and $t \in C_{G^*}(s)_\ell$.
Let $L^*$ be the minimal Levi subgroup containing $C_{G^*}(st)$. 
Then $t \in Z(L^*)$ if one of the following conditions is satisfied: 
\begin{enumerate}[label=(\alph*)]
\item $\ell$ is good for $L^*$ and $C_{G^*}(t)$ is connected;
\item $\ell$ is good for $L^*$ and $C_{G^*}(st)$ is connected;
\item $\ell$ is good for $C_{G^*}(s)$, the order of $s$ is not divisible by any bad primes for $L^*$ and $C_{G^*}(st)$ is connected.
\end{enumerate}
\end{lem} 

\begin{proof}
Suppose (a) is satisfied. We have
\begin{align*}
C_{G^*}(st) \subseteq L^* \cap C_{G^*}(t) \subseteq C_{G^*}(t).
\end{align*} 
By \cite[Proposition 2.1]{Geck-Hiss} and the fact that $C_{G^*}(t)$ is connected, $C_{L^*}(t)=L^* \cap C_{G^*}(t)$ is a Levi subgroup of $G^*$. The minimality of $L^*$ yields $L^*=L^* \cap C_{G^*}(t)$. In other words, $L^* \subseteq C_{G^*}(t)$ which implies that $t \in Z(L^*)$. 

Assume condition (b) to be satisfied. 
We have
\begin{align*}
C_{G^*}(st) \subseteq L^* \cap C_{G^*}^\circ(t) \subseteq C_{G^*}^\circ(t).
\end{align*}
By \cite[Proposition 2.1]{Geck-Hiss}, $ L^* \cap C_{G^*}^\circ(t)=C_{L^*}^\circ(t)$ is a Levi subgroup of $L^*$. As a Levi subgroup of a Levi subgroup of $G^*$, $C_{L^*}^\circ(t)$ is a Levi subgroup of $G^*$ itself. Since $C_{G^*}(st) \subseteq C_{L^*}^\circ(t)$, the minimality of $L^*$ implies $L^*=C_{L^*}^\circ(t)$; in other words $t \in Z(L^*)$.

Assume condition (c) to be satisfied. We claim that $L^*=C_{G^*}(st)$. Since $s$ and $t$ are commuting elements of coprime order, we have
\begin{align*}
C_{G^*}(st)=C_{C_{G^*}(s)}(t)=C_{C_{G^*}(t)}(s).
\end{align*}
In particular, $C_{C_{G^*}(s)}(t)$ and $C_{C_{G^*}(t)}(s)$ are connected. By our assumption on the order of $s$, $C_{G^*}(st)=C_{C_{G^*}^\circ(s)}(t)$ is a Levi subgroup of $C_{G^*}^\circ(s)$ (see \cite[Proposition 2.1]{Geck-Hiss}).
Additionally, 
\begin{align*}
C_{G^*}(st) \subseteq L^* \cap C_{G^*}^\circ(s) \subseteq C_{G^*}^\circ(s).
\end{align*} 
By \cite[Proposition 1.10]{Ruwen} and our assumption on the order of $s$, $L^* \cap C_{G^*}^\circ(s)$ is a Levi subgroup of $C_{G^*}^\circ(s)$. The minimality of $L^*$ therefore yields $C_{G^*}(st)=L^* \cap C_{G^*}^\circ(s)=C_{L^*}^\circ(s)$. 
Applying \cite[Proposition 2.1]{Geck-Hiss} again, we see that $C_{L^*}^ \circ(s)$ is a Levi subgroup of $L^*$. Hence, $C_{L^*}^\circ(s)$ is a Levi subgroup of $G^*$ as well. Now, the minimality of $L^*$ implies $L^*=C_{L^*}^\circ(s)=C_{G^*}(st)$ which proves the claim. In particular, $t \in Z(L^*)$.
\end{proof} 

When applying Lemma \ref{t im Zentrum}, we will always work with $F$-stable elements. A natural question is therefore if the minimal Levi subgroup $L^*$ in Lemma \ref{t im Zentrum} is $F$-stable as well. When $C_{G^*}(st)$ is connected this is immediate, since we then have $L^*=C_{G^*}(Z(C_{G^*}(st))^\circ)$. When $C_{G^*}(st)$ is disconnected, we use the following result. 

First, we need some notation. If $s \in G^{*F}$, we denote the $G^*$-conjugacy class of $s$ by $(s)$ and its $G^{*F}$-conjugacy class by $[s]$. If $Z(G)$ is disconnected, $(s)$ can split into multiple  $G^{*F}$-conjugacy classes. This phenomenon is well understood (see e.g. \cite[Theorem 2.1.5 (b)]{GLS} where we set $\Omega=(s)$)

\begin{lem}\label{F-stable Levis} Let $G$ be a connected reductive group defined over $\mathbb{F}_q$ with Frobenius endomorphism $F: G \to G$. Let $s \in G^{*F}$ be a semisimple element and let $L^* \subseteq G^*$ be the minimal Levi subgroup containing $C_{G^*}(s)$. Then $L^*$ is $F$-stable.
\end{lem}
\begin{proof}
Let 
\begin{align*}
V:=\{(t,M^*) \text{ } | \text{ } t \in (s) \text{ and } M^* \text{ is the minimal Levi subgroup containing } C_{G^*}(t) \}.
\end{align*}
Now $G^*$ acts transitively on $V$ by conjugation and this action is compatible with the natural $F$-action on $V$. Hence  there exists an $F$-stable pair $v=(z,L^*) \in V^F$ by \cite[Theorem 21.11 (a)]{Malle-Testerman}. The stabilizer of $v$ is $G^{*}_v=C_{G^*}(z) \cap N_{G^*}(L^*)=C_{G^*}(z)$. By \cite[Theorem 21.11]{Malle-Testerman} we therefore have a natural 1-1 correspondence
\begin{align*}
\{G^{*F}\text{-orbits on } V^F \} \longleftrightarrow \{G^{*F} \text{-classes in } (s)\}.
\end{align*}
In particular, there is a $G^{*F}$-orbit of $F$-stable Levi subgroups corresponding to $[s]$. It follows that the minimal Levi containing $C_{G^*}(s)$ is $F$-stable.
\end{proof}

We denote the characteristic function of the set of $\ell$-regular elements of $G^F$ by $\gamma_{\ell'}$.

\begin{thm}\label{restriction}
Let $G$ be a connected reductive group defined over $\mathbb{F}_q$ with Frobenius endomorphism $F:G \to G$. Let $\ell$ be a bad prime for $G$. Let $s \in G^{*F}$ be a semisimple $\ell'$-element and $t \in C_{G^*}(s)_\ell^F$. Let $L^*$ be the minimal ($F$-stable) Levi subgroup containing $C_{G^*}(st)$. If $\chi \in \mathcal{E}(G^F,st)$ then $\chi^\circ \in \mathbb{Q}\hat{\mathcal{E}}(G^F,s)$ if one of the following conditions is satisfied:
\begin{enumerate}[label=(\alph*)]
\item $t \in Z(L^*)$;
\item $C_{G^*}^\circ(st)$ is of type $A$ and $C_{G^*}(st)^F=C_{G^*}^\circ(st)^F$. 
\end{enumerate}
\end{thm}
\begin{proof}
If $t=1$, then $\chi^\circ \in \mathcal{E}(G^F,s)$ and we are done. Hence, assume $t \neq 1$. Let $L$ be an $F$-stable Levi subgroup of $G$ dual to $L^*$ (see \cite[Definition 13.10]{Digne-Michel}). By \cite[Theorem 9.16]{Cabanes-Enguehard} there is a character $\pi \in \mathcal{E}(L^F,st)$ such that $\chi=\epsilon_G \epsilon_L R^G_L(\pi)$. 

Suppose that (a) is satisfied. Using a slight variation of the proof of \cite[Theorem 3.1]{Geck-Hiss}, we show that $\chi^\circ \in \mathbb{Z}\hat{\mathcal{E}}(G^F,s)$. By \cite[Theorem 9.16]{Cabanes-Enguehard} there is a character $\pi \in \mathcal{E}(L^F,st)$ such that $\chi=\epsilon_G \epsilon_L R^G_L(\pi)$. Since $t \in Z(L^*)^F$, there exists a character $\theta_t$ of $L^F$, dual to $t$, such that $\pi=\theta_t\lambda$, where $\lambda \in \mathcal{E}(L^F,s)$ (see \cite[Proposition 13.30]{Digne-Michel}). The order of $\theta_t$ is equal to the order of $t$ and is therefore a power of $\ell$. Thus, $\theta_t^\circ=1_{L^F}^\circ$. We have
\begin{align*}
\chi \gamma_{\ell'} 
&= \epsilon_G \epsilon_L R^G_L(\theta_t \lambda) \gamma_{\ell'} \\
&=\epsilon_G \epsilon_L R^G_L(\theta_t \lambda \gamma_{\ell'})  \text{ } \text{ } \text{(\cite[Proposition 3.8]{Digne-Michel2})} \\
&=\epsilon_G \epsilon_L R^G_L(\theta_t^\circ \lambda^\circ) \\ 
&=\epsilon_G \epsilon_L R^G_L(\lambda^\circ) \\
&=\epsilon_G \epsilon_L R^G_L(\lambda) \gamma_{\ell'}.
\end{align*}
By \cite[Corollary 6]{Lusztig} every irreducible constituent of $R^G_L(\lambda)$ lies in $\mathcal{E}(G^F,s)$. Since $\chi^\circ=(\chi \gamma_{\ell'})_{G^F_{\ell'}}$, it follows that $\chi^\circ \in \mathbb{Z} \hat{\mathcal{E}}(G^F,s)$.

Now, suppose that condition (b) is satisfied. As $C_{G^*}^\circ(st)$ is of type $A$ and $C_{G^*}(st)^F=C_{G^*}^\circ(st)^F$, every irreducible character in $\mathcal{E}(G^F,st)$ is uniform (see \cite[Definition 12.11]{Digne-Michel}). We can therefore write 
\begin{align*}
\chi= \sum_{T^* \subseteq C_{G^*}^\circ(st)} \alpha_{T^*} R_{T^*}^G(st),
\end{align*}
where $T^*$ runs over the $F$-stable maximal tori of $C_{G^*}^\circ(st)$ with suitable coefficients $\alpha_{T^*} \in \mathbb{Q}$. If we restrict $\chi$ to the $\ell$-regular elements of $G^F$, 
we see that
\begin{align*}
\chi^\circ=\sum_{T^* \subseteq C_{G^*}^\circ(st)} \alpha_{T^*} R_{T^*}^G(s)^\circ
\end{align*}
because $R_{T^*}^G(st)^\circ=R_{T^*}^G(s)^\circ$ by \cite[Proposition 2.2]{Hiss}. Since $\alpha_{T^*} \in \mathbb{Q}$ for every $F$-stable maximal torus $T^* \subseteq C_{G^*}^ \circ(st)$ (see \cite[Proposition 12.12]{Digne-Michel}), it follows that $\chi^\circ$ is a $\mathbb{Q}$-linear combination of the characters in $\hat{\mathcal{E}}(G^F,s)$. This proves the assertion.
\end{proof}
Now, we will introduce the notation and tools needed to prove Theorem \ref{Theorem B} and its corollary. A semisimple element $s$ of $G$ is called \textbf{quasi-isolated} if its centraliser $C_G(s)$ is not contained in any proper Levi subgroup of $G$. If, moreover, even $C_G^\circ(s)$ is not contained in any proper Levi subgroup, $s$ is called \textbf{isolated}. These elements have been classified by Bonnafé in \cite{Bonnafe}. We recall the result for simple groups of exceptional type.

\begin{prop}[Bonnafé] \label{Bonnafe}
Let $G$ be a simple, exceptional algebraic group of adjoint type. Then the conjugacy classes of semisimple, quasi-isolated elements of $G$, their orders, the root system of their centraliser $C_G(s)$, and the group of components $A(s):=C_G(s)/C_G^\circ(s)$ are as given in Table 1. 
\end{prop}
\noindent
The order of $s$ is denoted by $o(s)$.
\small
\begin{center}  \captionof{table}{Quasi-isolated elements in exceptional groups}
\begin{longtable}{|c|c|l|c|c|}
 \hline 
$G$	&	$o(s)$		&	$C_G^\circ(s)$	&	$A(s)$	&	isolated? \\ \hline  \hline
$G_2$ &   2	& 	$A_1 \times A_1$ & 1 & yes \\
	&	3	&	$A_2$ 	&	1	&	yes \\
	
$F_4$	&	2	&	$C_3 \times A_1, B_4$	&	1	&	yes\\
	&	3	&	$A_2 \times A_2$	&	1 	&	yes\\
	&	4	&	$A_3 \times A_1$	&	1 	&	yes\\
$E_6$	&	2	&	$A_5 \times A_1$	&	1 &	 yes\\
	&	3	&	$A_2 \times A_2 \times A_2$	&	3	&	yes\\
	&	3	&	$D_4$	&	3	&	no \\
	&	6	&	$A_1 \times A_1 \times A_1 \times A_1$	&	3 	&	no \\
$E_7$ &	2	&	$D_6 \times A_1$	&	1 &	yes\\
	&	2	&	$A_7$	&	2	&	yes \\
	&	2	&	$E_6$	&	2	&	no \\
	&	3	&	$A_5 \times A_2$	&	1 	&	yes\\
	&	4	&	$A_3 \times A_3 \times A_1$	&	2	&	yes \\
	&	4	&	$D_4 \times A_1 \times A_1$	&	2	&	no	\\
	&	6	&	$A_2 \times A_2 \times A_2$	&	2 	&	no \\
$E_8$	&	2	&	$D_8, E_7 \times A_1$	&	1	&	yes \\
	&	3	&	$A_8, E_6 \times A_2$	&	1 	&	yes	\\
	&	4	&	$D_5 \times A_3, A_7 \times A_1$	&	1	&	yes \\
	&	5	&	$A_4 \times A_4$	&	1	&	yes \\
	&	6	&	$A_5 \times A_2 \times A_1$	&	1	&	yes	 \\ \hline
\end{longtable}
\label{tabu:quasi-isolated elements}
\end{center}
\normalsize
\begin{definition}
(a) The $\ell$-blocks contained in $\mathcal{E}_\ell(G^F,s)$  for a semisimple, quasi-isolated $\ell'$-element $s \in G^{*F}$ are  called \textbf{quasi-isolated}. If $s=1$ they are also called \textbf{unipotent}.
(b) Let $H=G^F/Z$, for some subgroup $Z \subseteq Z(G^F)$. A block of $H$ is said to be \textbf{quasi-isolated} if it is dominated by a quasi-isolated block of $G^F$ and \textbf{unipotent} if it is dominated by a unipotent block of $G^F$.
\end{definition}
\noindent

Let $e \geq 1$ be an integer. We say an irreducible character of $G^F$ is \textbf{$e$-cuspidal} if $^*\!R_L^G(\chi)=0$ for every proper $e$-split Levi subgroup $L \subsetneq G$. Let $\lambda \in \operatorname{Irr}(L^F)$ for an $e$-split Levi subgroup $L \subseteq G$. Then we call $(L, \lambda)$ an \textbf{$e$-split pair}. We define a binary relation on $e$-split pairs by setting $(M, \zeta) \leq_e (L, \lambda)$ if $M \subseteq L$ and $\langle ^*\!R_M^L(\lambda), \zeta \rangle \neq 0$. Since the Lusztig restriction of a character is in general not a character, but a generalized character, the relation $\leq_e$ might not be transitive. We denote the transitive closure of $\leq_e$ by $\ll_e$. If $(L, \lambda)$ is minimal for the partial order $\ll_e$, we call $(L, \lambda)$ an \textbf{$e$-cuspidal pair} of $G^F$. Moreover, we say $(L, \lambda)$ is a \textbf{proper} $e$-cuspidal pair if $L \subsetneq G$ is a proper $F$-stable Levi subgroup of $G$. 

The $e$-cuspidal pairs of $G^F$ are the key ingredient in the parametrization of the blocks of $G^F$. We write $\mathcal{E}(G^F,(L,\lambda)):=\{\chi \in \text{Irr}(G^F) \mid (L,\lambda) \leq_e (G, \chi)\}$. \\

\noindent
From now on let $G$ be a simple, simply connected algebraic group of exceptional type defined over $\mathbb{F}_q$ with Frobenius endomorphism $F: G \to G$. Furthermore, suppose that $\ell \nmid q$ is a bad prime for $G$. Let $\mathcal{E}(G^F, \ell') := \bigcup_{\ell' \textit{-elements } s \in G^{*F}} \mathcal{E}(G^F,s)$ denote the union of Lusztig series corresponding to semisimple $\ell'$-elements of $G^{*F}$. 

Let $\chi \in \mathcal{E}(G^F, \ell')$. We say that $\chi$ is of \textbf{central $\ell$-defect} if $|G^F|_\ell=\chi(1)_\ell |Z(G)^F|_\ell$ and we say that $\chi$ is of \textbf{quasi-central $\ell$-defect} if some constituent of $\chi_{[G,G]^F}$ is of central $\ell$-defect. \\

\noindent
Using $e$-cuspidal characters of central $\ell$-defect, Enguehard was able to parametrise the unipotent blocks of $G^F$ for bad $\ell$ (see \cite{Enguehard}). Later on, Kessar and Malle, used the characters of quasi-central $\ell$-defect to parametrise the quasi-isolated blocks of $G^F$ for bad $\ell$ (see \cite{Malle-Kessar}). For this, they had to prove that the relation $\leq_e$, restricted to the set of $e$-cuspidal pairs $(L, \lambda)$ corresponding to a quasi-isolated element $1 \neq s \in G^{*F}$, is transitive (\cite[Theorem 1.4 (a)]{Malle-Kessar}). In particular, Malle and Kessar showed that $R_L^G$ satisfies an $e$-Harish-Chandra theory above each $e$-cuspidal pair $(L, \lambda)$ corresponding to a quasi-isolated element $1 \neq s \in G^{*F}$ (\cite[Theorem 1.4 (b)]{Malle-Kessar}). \\

\noindent
The reason we focus our attention on the quasi-isolated blocks of $G^F$ are the results of Bonnafé--Rouquier \cite{Bonnafe-Rouquier} and more recently Bonnafé--Dat--Rouquier \cite{BDR}. We use their reduction to quasi-isolated blocks to later prove the corollary to Theorem \ref{Theorem B}.  
\begin{notation} \label{notation} Let $B$ be a quasi-isolated $\ell$-block of $G^F$ where $\ell$ is bad. In this case $B$ corresponds to (exactly) one of the numbered lines in the tables of \cite{Enguehard} or \cite{Malle-Kessar}. If $i$ is that number, we will say that $B$ is a block \textbf{numbered} $i$ or that $B$ is \textbf{of type} $i$. Moreover, if $B$ is unipotent we add $u$ as a subscript and say that $B$ is a block numbered $i_u$ or of type $i_u$. 
\end{notation}

Next, we will briefly recall the Malle--Robinson conjecture. Let $H$ be a finite group. If $N \unlhd K \subseteq H$ are two subgroups of $H$, we call the quotient $K/N$ a \textbf{section} of $H$. The \textbf{sectional $\ell$-rank} $s(H)$ of a finite group $H$ is then  defined to be the maximum of the ranks of elementary abelian $\ell$-sections of $H$. Note that $s(K/N) \leq s(H)$ for every section  $K/N$ of $H$. \\

\noindent
For a block $B$ of $H$ we denote the number of irreducible Brauer characters in $B$ by $l(B)$.

\begin{conjecture}[Malle--Robinson, {\cite[Conjecture 1]{Malle-Robinson}}]\label{Malle-Robinson Conjecture}
Let $B$ be an $\ell$-block of a finite group $H$ with defect group $D$. Then 
\begin{align*}
l(B) \leq \ell^{s(D)}.
\end{align*}
\end{conjecture}
\noindent
If strict inequality holds, we say that the conjecture holds in \textit{strong form}. Since the defect groups of a given block $B$ are conjugate and therefore isomorphic to each other, we often write $s(B)$ instead of $s(D)$. 
\section{The blocks of $G_2(q)$ and $F_4(q)$}
\noindent
Let $G$ be a simple, simply connected algebraic group of type $G_2$ or $F_4$ defined over $\mathbb{F}_q$ with Frobenius endomorphism $F: G \to G$.
The groups of type $G_2$ and $F_4$ are self-dual and therefore also of adjoint type.
As a result, centralisers of semisimple elements in $G$ are connected.

The ranks of these groups are small enough to extend the assertion of Theorem \ref{Theorem A} to all $\ell$-blocks.

%
%

\begin{thm} \label{generating set F4}
Let $\ell \nmid q$ be a bad prime for $G$ and let $s \in G^{*F}$ be a semisimple $\ell'$-element. Then $\hat{\mathcal{E}}(G^F,s) \cup (\bigcup_t \hat{\mathcal{E}}(G^F,st)$ generates  $\mathbb{Q}\operatorname{IBr}(\mathcal{E}_\ell(G^F,s))$,  where $t$ runs over the $\ell$-elements of $C_{G^*}(s)^F$ such that 
\begin{enumerate}
\item $st$ is quasi-isolated; and
\item  $C_{G^*}(st)$ is not of type $A$.
\end{enumerate} 
In particular, $l_s \leq |\mathcal{E}(G^F,s)|+ \sum_{1 \neq t}|\mathcal{E}(G^F,st)|$.
\end{thm}
\begin{proof}
If $t \in C_{G^*}(s)_\ell^F$ satisfies conditions (1) and (2) we can not apply Theorem \ref{restriction}, which is why the corresponding Lusztig series are part of the asserted generating set. 

Since $\{\chi^\circ \mid \chi \in \operatorname{Irr}(\mathcal{E}_\ell(G^F,s)) \}$ generates $\mathbb{Q}\operatorname{IBr}(\mathcal{E}_\ell(G^F,s))$ by \cite[(3.16) Lemma]{Navarro}, it suffices to show that $\chi^\circ \in \mathbb{Q}\hat{\mathcal{E}}(G^F,s)$ for every $\chi \in \mathcal{E}(G^F,st))$, where $1 \neq t \in C_{G^*}(s)_\ell^F$ does not satisfy conditions (1) or (2). 
Suppose $1 \neq t \in C_{G^*}(s)_\ell^F$ such that $st$ is not quasi-isolated, i.e. condition (1) is not satisfied. Let $\psi \in \mathcal{E}(G^F,st)$. Let $L^* \subsetneq G^*$ be the minimal Levi subgroup containing $C_{G^*}(st)$. In particular, $L^*$ is of classical type and $st$ is a quasi-isolated element of $L^*$. If $L^*$ is a torus, we are done. If $L^*$ is of type $A$ then $t \in Z(L^*)$ by Lemma \ref{t im Zentrum} (a). Thus, $\psi^\circ \in \mathbb{Q}\hat{\mathcal{E}}(G^F,s)$ by Theorem \ref{restriction}(a). Now, suppose that $L^*$ is of type $B$ or $C$ (hence $G$ is of type $F_4$). By \cite{Chevie} or \cite{Bonnafe}, we know that either $st \in Z(L^*)$ or that $C_{G^*}(st)$ is of type $A$. In the first case, it is immediate that $t \in Z(L^*)$ and, therefore, $\psi^\circ \in \mathbb{Q}\hat{\mathcal{E}}(G^F,s)$ by Theorem \ref{restriction} (a) while in the second case, $\psi^\circ \in \mathbb{Q}\hat{\mathcal{E}}(G^F,s)$ follows from Theorem \ref{restriction} (b). Suppose that $1 \neq t \in C_{G^*}(s)_\ell^F$ does not satisfy condition (2). Then $\psi^\circ \in \mathbb{Q}\hat{\mathcal{E}}(G^F,s)$ by Theorem \ref{restriction} (b).
\end{proof}
\noindent

\begin{remark}

Let $s \in G^{*F}$ be a semisimple quasi-isolated $\ell'$-element. Note that for $1 \neq t \in C_{G^*}(s)_\ell^F$, by Table \ref{tabu:quasi-isolated elements}, $st$ is only quasi-isolated when $s=1$. More precisely, $st$ satisfies both conditions (1) and (2) only when $G$ is of type $F_4$, $s=1$, $\ell=2$ and $t \in C_{G^*}(s)_2^F$ such that $C_{G^*}(t)^F=B_4(q)$ or $C_3(q)A_1(q)$.
\end{remark}
\noindent
It follows from Theorem \ref{generating set F4} that, if $B$ is an $\ell$-block contained in $\mathcal{E}_\ell(G^F,s)$ then a generating set for $\mathbb{Q} \operatorname{IBr}(B)$ is given by $\{ \chi^\circ \mid \chi \in \operatorname{Irr}(B) \cap (\mathcal{E}(G^F,s) \cup (\bigcup_{1 \neq t} \mathcal{E}(G^F,st)))\}$. Let $c(B):=|\operatorname{Irr}(B) \cap (\mathcal{E}(G^F,s) \cup (\bigcup_{1 \neq t} \mathcal{E}(G^F,st)))|$. To prove the Malle--Robinson conjecture for the quasi-isolated blocks $B$ of $G^F$, we show that
\begin{align*}
l(B) \leq c(B) \leq \ell^{s(B)}.
\end{align*} 
We define
\begin{align*}
e=e_\ell(q):=\text{order of } q \text{ modulo }\begin{cases}
\ell \text{ } \text{if} \text{ } \ell>2, \\
4 \text{ } \text{if} \text{ } \ell=2.
\end{cases}
\end{align*} 
Since $\ell$ is assumed to be a bad prime, the only cases that occur are $e=1$ and $e=2$.

\begin{thm} \label{Conjecture F4}
Let $G$ be a simple, simply connected algebraic group of type $G_2$ or $F_4$ defined over $\mathbb{F}_q$ with Frobenius endomorphism $F: G \to G$.
Let $\ell \nmid q$ be a bad prime for $G$. Then the Malle--Robinson conjecture holds for  the quasi-isolated $\ell$-blocks of $G^F$.
\end{thm}


\begin{proof}
We will prove the assertion for the harder case, namely when $G$ is of type $F_4$. The proof for type $G_2$ follows the same approach.
Let $B$ be a non-unipotent quasi-isolated block of $G^F$ associated to a line in \cite[Table 2]{Malle-Kessar} and let $(L_1, \lambda_1),
\dots,$ $(L_r, \lambda_r)$ be the $e$-cuspidal pairs associated to that block. By Theorem \ref{generating set F4} and \cite[Theorem 1.4]{Malle-Kessar} we conclude that 
\begin{align*}
c(B)=\sum_{i=1}^r |\mathcal{E}(G^F,(L_i,\lambda_i))|.
\end{align*}
Since $R_L^G$ satisfies an $e$-Harish-Chandra theory above each $(L_i, \lambda_i)$ by \cite[Theorem 1.4]{Malle-Kessar}, 
\begin{align*}
|\mathcal{E}(G^F,(L_i, \lambda_i))|=|\text{Irr}(W_{G^F}(L_i, \lambda_i))|,
\end{align*}
and these relative Weyl groups can be found in \cite[Table 2]{Malle-Kessar}. Let $(L, \lambda)$ now be the unique pair parametrising $B=b_{G^F}(L, \lambda)$ by \cite[Theorem 1.2]{Malle-Kessar}. Let $D$ be a defect group of $B$. By \cite[Theorem 1.2 (b)]{Malle-Kessar}, $Z(L)_\ell^F \subseteq D$ and hence $s(Z(L)_\ell^F) \leq s(D)$. We prove the Malle--Robinson conjecture by establishing the stronger inequality 
\begin{align*}
l(B) \leq c(B) < \ell^{s(Z(L)_\ell^F)} \leq \ell^{s(B)}.
\end{align*}
Checking Table \cite[Table 2]{Malle-Kessar}, we see that $Z(L)^F=\Phi_e^k$ is an $e$-torus in every case.  Let $\ell=3$ and let $B$ be a quasi-isolated 3-block. If $Z(L)^F=\Phi_e^k$, then $s(Z(L)_3^F)=k$. The $k$'s can be read off from Table \cite[Table 2]{Malle-Kessar} and we see that $c(B) < 3^{s(Z(L)_3^F)}$ in every case. Now, let $\ell=2$. Let $B=b_{G^F}(L, \lambda)$ be the $2$-block corresponding to line 1 of \cite[Table 2]{Malle-Kessar}. To prove the conjecture it suffices to take $s(Z(L)_2^F)$ again. Let $B=b_{G^F}(L, \lambda)$ now be the block corresponding to line 2 of \cite[Table 2]{Malle-Kessar}. To prove the conjecture we use line 2b of Table \cite[Table 2]{Malle-Kessar}. As seen in the proof of \cite[Proposition 3.5]{Malle-Kessar}, the $1$-Harish-Chandra series corresponding to line 2b lies  in $B$. By \cite[Proposition 2.17]{Malle-Kessar}, $Z(M)_2^F=\Phi_2^4 \subseteq D$ where $(M, \zeta)$ is the pair of line 2b. Note that $\Phi_2$ is always divisible by 2 unless $q$ is a power of 2. Since we are working in cross-characteristic and assume $\ell=2$, this can not be the case. Hence, $Z(M)_2^F=\Phi_2^4$ yields an elementary abelian 2-subgroup of rank 4. It follows that
\begin{align*}
l(B) \leq c(B) < 2^{s(Z(L)_2^F)} \leq 2^{s(B)}.
\end{align*}
If $e=2$, then the Ennola dual of line 2b gives a 1-split torus $\Phi_1^4$ which yields an elementary abelian 2-subgroup of rank 4. The rest of the proof did not depend $e$. 

Now, let $B$ be a unipotent block of $G^F$ with defect group $D$. Let $(L_1, \lambda_1),
\dots,$ $(L_r, \lambda_r)$ be the unipotent $e$-cuspidal pairs associated to $B$ by \cite[Théorème A (a)]{Enguehard}. We know that $r=1$ whenever $B$ is not the principal $\ell$-block by \cite[Théorème A (c)]{Enguehard}. Let $\ell=3$. In this case \[c(B)= |\operatorname{Irr}(B) \cap \mathcal{E}(G^F,s)|= \sum_{i=1}^r |\mathcal{E}(G^F,(L_i, \lambda_i)|. \]
This sum can be computed using Chevie \cite{Chevie}. Suppose that $(L,\lambda)$ is the unipotent $e$-cuspidal pair such that $B=b_{G^F}(L,\lambda)$. Then $Z(L)_3^F \subseteq D$ and we see that $c(B) < 3^{s(Z(L)_3^F)}$. Hence the Malle--Robinson conjecture holds. Now, let $\ell=2$. In this case, $\mathbb{Q}\operatorname{IBr}(\mathcal{E}_2(G^F,1))$ is generated by $\mathcal{E}(G^F,1) \cup \mathcal{E}(G^F,t_1) \cup \mathcal{E}(G^F,t_2)$, where $t_1$ and $t_2$ are quasi-isolated 2-elements of $G^*$ with $C_{G^*}(t_1)=C_3 \times A_1$ and $C_{G^*}(t_2)=B_4$. There are 3 different unipotent 2-blocks (see  \cite{Enguehard}) - the principal block, corresponding to a maximal torus, and two blocks of defect zero corresponding to the unipotent $e$-cuspidal characters of $G^F$. If $B=b_{G^F}(G, \chi)$ is one of the blocks of defect zero then $\operatorname{Irr}(B)=\{\chi \}$. Hence $l(B)=c(B)$ and the Malle--Robinson conjecture clearly holds. Furthermore, it follows that every Lusztig series of the form $\mathcal{E}(G^F,t)$ where $1 \neq t \in G^{*F}_2$, lies in the principal block. Suppose that $B$ is the principal block. Here, the conjectured upper bound has been proved in \cite[Proposition 6.10]{Malle-Robinson}.
\end{proof}

\section{The quasi-isolated blocks of $E_6(q)$ and $^2\!E_6(q)$}
\noindent
Let $G$ be a simple, simply connected algebraic group of type $E_6$ defined over $\mathbb{F}_q$ with Frobenius endomorphism $F:G \to G$. Then $G^F= E_{6,sc}(q)$ or $^2\!E_{6,sc}(q)$ and the dual group $G^*$ (which is of adjoint type) contains semisimple elements whose centralisers are disconnected. However, recall that centralisers of $3'$-elements are connected, by \cite[Proposition 14.20]{Malle-Testerman}.\\



\noindent
We know a great deal about the Levi subgroups of $G$.

\begin{lem} \label{Levis simply-connected} Let $L^* \subseteq G^*$ be a proper Levi subgroup of $G^*$.  Then $[L^*,L^*]$ is simply connected unless $L^*$ is of type $A_2^2$, $A_2^2 \times A_1$ or $A_5$.
\end{lem}

\begin{proof}
This can be checked with Chevie \cite{Chevie}.
\end{proof}

\begin{remark} \label{quasi-isolated element of order 6 in E6}
It can be checked that every quasi-isolated element $z \in G^*$ of order 6 is of the form $z=st$ where $s$ is quasi-isolated of order 3 with $C_{G^*}^\circ(s)$ of type $D_4$, and $t$ is quasi-isolated of order 2 with $C_{G^*}(t)$ of type $A_5 \times A_1$ (or vice-versa). 
\end{remark}

\begin{prop} \label{generating set quasi-isolated of order 3}
Let $\ell=2$ and let $s \in G^{*F}$ be a quasi-isolated semisimple $2'$-element. Then $\hat{\mathcal{E}}(G^F,s) \cup (\bigcup_t \hat{\mathcal{E}}(G^F,st)$ generates  $\mathbb{Q}\operatorname{IBr}(\mathcal{E}_2(G^F,s))$,  where $t$ runs over the $2$-elements of $C_{G^*}(s)^F$ such that 
\begin{enumerate}
\item $st$ is quasi-isolated; and
\item $C_{G^*}(st)^F \neq C_{G^*}^\circ(st)^F$ or $C_{G^*}^\circ(st)$ is not of type $A$.
\end{enumerate} 
In particular, $l_s \leq |\mathcal{E}(G^F,s)|+ \sum_{1 \neq t}|\mathcal{E}(G^F,st)|$.
\end{prop}

\begin{proof}
We proceed as in the proof of Theorem \ref{generating set F4}. Let $1 \neq t \in C_{G^*}(s)_2^F$ be such that $st$ is not quasi-isolated in $G^*$. Let $\psi \in \mathcal{E}(G^F,st)$. Let $L^* \subsetneq G^*$ be the minimal Levi subgroup containing $C_{G^*}(st)$. Note that the proper Levi subgroups of $G^*$ are either of type $D$ or a product of groups of type $A$ (or maximal tori, in which case $t \in Z(L^*)$). Moreover, $C_{G^*}(t)$ is connected as $t$ is a $3'$-element. If $L^*$ is of type $A$ then $\psi^\circ \in \mathbb{Z}\hat{\mathcal{E}}(G^F,s)$ by Lemma \ref{t im Zentrum} (a) and the proof of Theorem \ref{restriction}. Now, suppose that $L^*$ is of type $D$. Since $[L^*,L^*]$ is simply connected by Lemma \ref{Levis simply-connected}, $C_{G^*}(st)=C_{L^*}(st)$ is connected. Therefore $\psi^\circ \in \mathbb{Q}\hat{\mathcal{E}}(G^F,s)$ follows from Lemma \ref{t im Zentrum} (c) and the proof of Theorem \ref{restriction}. Suppose that $1 \neq t \in C_{G^*}(s)_2^F$ does not satisfy condition (2). Then $\psi^\circ \in \mathbb{Q}\hat{\mathcal{E}}(G^F,s)$ by Theorem \ref{restriction} (b). Hence the assertion is proved.
\end{proof}
\noindent
To get an upper bound for $l_s$ where $s \in G^{*F}$ corresponds to the blocks numbered 14 and 15, we will use a slightly different approach. Note that, by \cite[Table 24.2]{Malle-Testerman}), $Z(G^F)=1$ if 
\begin{enumerate}[label=(\alph*)]
\item $G^F=E_{6,sc}(q)$ and $3 \nmid (q-1)$, or 
\item $G^F=\text{}^2\!E_{6,sc}(q)$ and $3 \nmid (q+1)$.
\end{enumerate}
In particular, $G_{ad}^F \cong G_{sc}^F$. The assertion of Theorem \ref{Theorem A} for a semisimple, quasi-isolated $s \in G^{*F}$ corresponding to the blocks 14, 15 therefore follows from the following result.

\begin{thm} \label{E_6 adjoint type}  Let $G$ be simple of adjoint type $E_6$ defined over $\mathbb{F}_q$ with Frobenius endomorphism $F: G \to G$. Let $\ell \nmid q$ be a bad prime for $G$ and let $s \in G^{*F}$ be a semisimple $\ell'$-element. Then $\hat{\mathcal{E}}(G^F,s)$ is a generating set of $\mathbb{Q}\operatorname{IBr}(\mathcal{E}_\ell(G^F,s))$. In particular, $l_s \leq |\mathcal{E}(G^F,s)|$
\end{thm}

\begin{proof} Note that there are no quasi-isolated elements in $G^*$ of order greater than 6. Further, the quasi-isolated elements of $G^*$ of order 6 are of the form $xy$ where $x$ is a semisimple quasi-isolated element of order 2 and $y$ is a generator for $Z(G^{*F}$. In particular, $C_{G^*}(xy)$ is of type $A$. Let $\ell=2$. If $1 \neq s \in G^{*F}$ is a semisimple quasi-isolated $2'$-element, then we can use the proofs of Section 3. Let $1 \neq s \in G^{*F}$ be a semisimple non-quasi-isolated $2'$-element and let $t \in C_{G^*}(s)_2^F$. Since $st$ is not quasi-isolated, there is a Levi subgroup $L^*$ minimal with respect to $C_{G^*}(st) \subseteq L^*$. Suppose $L^*$ is of type $A$, then $C_{G^*}(st)=C_{L^*}(st)$ is a Levi subgroup and hence $L^*=C_{G^*}(st)$. If $L^*$ is of type $D$, then $C_{G^*}(st)$ is of type $A$ and we can conclude as we did before. 

Suppose $\ell=3$ now. If $1 \neq s \in G^{*F}$ is a quasi-isolated $3'$-element, then $C_{G^*}(st)$ is of type $A$ for every $t \in C_{G^*}(s)_3^F$ since $C_{G^*}(st)=C_{C_{G^*}(s)}(t)$ and $C_{G^*}(s)$ is of type $A$. Let $1 \neq s \in G^{*F}$ be a semisimple non-quasi-isolated $3'$-element and let $t \in C_{G^*}(s)_3^F$. The assertion follows because 3 is a good prime for every proper Levi subgroup of $G^*$. 

Suppose that $1=s$ (so we are talking about the unipotent blocks). For $\ell \in \{2,3\}$ the centralisers of semisimple $\ell$-elements are either Levi subgroups of $G^*$ or of type $A$ and since $G$ is of adjoint type they are connected. Therefore we can conclude as before.
\end{proof}

\noindent

\begin{thm} \label{Conjecture E6}
Let $G$ be a simple, simply connected algebraic of type $E_6$ defined over $\mathbb{F}_q$ with Frobenius endomorphism $F: G \to G$ 
Then the Malle--Robinson conjecture holds for the quasi-isolated blocks of $G^F$ and $G^F/Z(G^F)$.
\end{thm}

Recall Notation \ref{notation}.

\begin{proof}
We demonstrate the proof for $G^F=E_6(q)$ and $E_6(q)/Z(E_6(q))$. The proofs for $^2\!E_6(q)$ and $^2\!E_6(q)/Z(^2\!E_6(q))$ are similar.

We can determine $c(B)$ by checking \cite[Table 3]{Malle-Kessar}. If $B=b_{G^F}(L, \lambda)$ is a unipotent block or a non-unipotent quasi-isolated block numbered 1, 2, 4, 5, 8, 9, 10, 11, 12, 14 or 15 then $s(Z(L)_\ell^F)$ suffices to establish the conjectured upper bound.
For the block $B$ numbered 7 we use the 1-cuspidal pair $(L, \lambda)$ in line 1 (see the proof of \cite[Proposition 4.3]{Malle-Kessar}). We have $L=C_G^\circ(Z(L)_\ell^F)$ and $\lambda$ is of central $\ell$-defect. By combining \cite[Proposition 2.13 (a)]{Malle-Kessar} and \cite[Proposition 2.16 (3)]{Malle-Kessar}, the pair $(L, \lambda)$ satisfies the conditions of \cite[Proposition 2.12]{Malle-Kessar}. Hence, $(Z(L)_2^F,b)$ is a $B$-Brauer pair where $b$ is the block of $L$ containing $\lambda$. In particular, $Z(L)_2^F= \Phi_1^6 \subseteq D$ where $D$ is a defect group of $B$.  For case 3 we use the 2-cuspidal pair from case 8. The blocks numbered 5, 11 and 13 have to be dealt with using completely different methods (see Proposition \ref{c(b) and s(B) exceptions} and Corollary \ref{Theorem A and B for E6 and ell=3}).

Let $\bar{B}$ now be a quasi-isolated block of $H=G^F/Z(G^F)$ with defect group $\bar{D}$ dominated by a quasi-isolated block $B$ of $G^F$ not of type 5, 11 and 13 (for these exceptions see Proposition \ref{c(b) and s(B) exceptions} and Corollary \ref{Theorem A and B for E6 and ell=3}). By \cite[Theorem (9.9)(c)]{Navarro}, $l(\bar{B})=l(B)$ and $\bar{D}$ is of the form $DZ(G^F)/Z(G^F)$, for a defect group $D$ of $B$. Suppose that $\ell=2$. Since $|Z(G^F)|$ is either 1 or 3, $D \cap Z(G^F)=\{1\}$. Thus, $DZ(G^F)$ is a direct product and it follows that $D Z(G^F)/Z(G^F) \cong D$, i.e. $s(\bar{D})=s(D)$. Thus, the conjecture holds for $\bar{B}$ since it holds for $B$. If $\ell=3$ and $e=2$ (i.e. $B$ is dominated by a block of type 14 or 15), then $Z(G^F)=1$. Thus, $\bar{B}=B$ and we are done. 
\end{proof}
\text{}

\begin{center}
\textsc{\large{The blocks numbered $5$ and $11$}} \\
\end{center}

\noindent
Let $G$ be a connected reductive algebraic group (only for this exposition) and let $G^*$ be a dual group. Let $W$ and $W^*$ be the Weyl groups of $G$ and $G^*$ respectively. By \cite[Proposition 4.2.3]{Carter} there is a natural isomorphism $W \cong W^*$. This isomorphism yields a canonical isomorphism between $N_{G^*}(L^*)/L^*$ and $N_G(L)/L$. Now, fix a semisimple $\ell'$-element s $\in G^{*F}$ and let $L^*=C_{G^*}(Z^\circ(C_{G^*}^\circ(s)))$ be the minimal Levi subgroup of $G^*$ containing $C_{G^*}^\circ(s)$. Furthermore set $N^*=C_{G^*}(s)^F.L^*$ and let $L$ be a dual of $L^*$ in $G$. Define $N$ to be the subgroup of $N_G(L)$ containing $L$ such that $N/L$ corresponds to $N^*/L^*$ via the canonical isomorphism between $N_{G^*}(L^*)/L^*$ and $N_G(L)/L$. \\

Let $\ell \nmid q$ be a prime. We denote the sum of the block idempotents of the $\ell$-blocks contained in $\mathcal{E}_\ell(G^F,s)$ and $\mathcal{E}_\ell(L^F,s)$ by $e_s^{G^F}$ and $e_s^{L^F}$ respectively. 

\begin{thm}[{Bonnafé-Dat-Rouquier, \cite[Theorem 7.7]{BDR}}] \label{BDR} Let the notations be as above. Then there exists a Morita equivalence
\begin{align*}
\mathcal{O}G^Fe_s^{G^F} \sim \mathcal{O}N^Fe_s^{L^F}
\end{align*}
together with a bijection $b \mapsto b'$ between the $\ell$-blocks of both sides, preserving defect groups and such that $\mathcal{O}G^Fb$ is Morita equivalent to $\mathcal{O}N^Fb'$.
\end{thm}

\begin{remark}
Recently a gap was found in the proof of the original result \cite[Theorem 7.7]{BDR}. However the problem arises only (in a very specific case) when groups of type $D_n$ ($n \geq 4$) are involved in $G$.
\end{remark}

Let $G$ now be a simple, simply connected algebraic group of type $E_6$ again. Let $s \in G^{*F}$ be a quasi-isolated element of order 3 with $C_{G^*}(s)^F=\Phi_3.^3\! D_4(q).3$ and let $B$ be of type 6 or 12 (see cite[Table 3]{Malle-Kessar}). We see that $L^*=C_{G^*}(s)^\circ$. Furthermore, $N/L$ is cyclic of order 3. Hence, we are in the situation of \cite[
Example 7.9]{BDR}. Thus, there exists a Morita equivalence $\mathcal{O}G^Fe_s^{G^F} \sim \mathcal{O}N^F e_1^{L^F}$ together with a bijection as in Theorem \ref{BDR} between the blocks on both sided that preserves defect groups and such that corresponding blocks are Morita equivalent. 
So every block contained in $\mathcal{E}_\ell(G^F,s)$ is Morita equivalent to a block of $N^F$ which itself covers a unipotent block of $L^F$.

\begin{remark}\label{covering BDR}
Let $H$ be a finite group and $K \unlhd H$. If $B$ is a block of $H$ covering a block $b$ of $K$ then $B$ has a defect group $D$ such that $D \cap K$ is a defect group of $b$ (see \cite[Theorem (9.26)]{Navarro}. We use this fact in the case where $H=N^F$ and $K=L^F$. 
\end{remark}


\begin{prop} \label{c(b) and s(B) exceptions}
Let $s \in G^{*F}$ be a quasi-isolated element of order 3 such that $C_{G^*}(s)^F=\Phi_3.^3\! D_4(q).3$. Let $B$ be of type 6 or 12. Then $c(B)=12$ and $4 \leq s(B)$. In particular, the Malle--Robinson conjecture holds in strong form for the blocks of type 5 and 11.
\end{prop}

\begin{proof}
We demonstrate the proof for the blocks of type 5. Let $B$ be a block numbered 5. By Proposition \ref{generating set quasi-isolated of order 3}, $c(B)=| \operatorname{Irr}(B) \cap \left( \mathcal{E}(G^F,s) \cup \mathcal{E}(G^F,st) \right)|$ where $t \in C_{G^*}(s)_2^F$ such that $C_{G^*}(st)^F=\Phi_3.A_1(q)A_1(q^3).3$. It can be shown that $|\operatorname{Irr}(B) \cap \mathcal{E}(G^F,s)|=8$ and $|\operatorname{Irr}(B) \cap \mathcal{E}(G^F,st)|=4$.
Hence, $c(B)=12$.

For the lower bound on $s(B)$ we use Theorem \ref{BDR} and the classification of unipotent blocks in bad characterstic obtained by Enguehard \cite{Enguehard}. Let $D$ be a defect group of $B$. We are interested in elementary abelian 2-sections of $D$. By Theorem \ref{BDR} we can reduce this to the study of defect groups of the Bonnafé--Dat--Rouquier correspondent block of $N^F$ which itself covers a unipotent block of $L^F$. By Remark \ref{covering BDR}, we are done if we can find a sufficiently large elementary abelian 2-section in the defect groups of this unipotent block of $L^F$. We can furthermore reduce this to the study of the defect groups of the unipotent blocks of the group $^3\!D_4(q)=[L,L]^F$ as can be seen as follows. Restriction of characters gives a bijection $\mathcal{E}(L^F,1) \to \mathcal{E}([L,L]^F,1)$ (see e.g. \cite[Proposition 13.20]{Digne-Michel}). By the character-theoretic characterization of covering blocks (see \cite[Theorem (9.2)]{Navarro}), we know that the unipotent blocks of $L^F$ cover the unipotent blocks of $[L,L]^F$.  By the classification of unipotent blocks in \cite{Enguehard}, the only unipotent 2-block of $^3\!D_4(q)$ is the principal block. So it is enough to show that the Sylow 2-subgroups of $^3\!D_4(q)$ have an elementary abelian section of order 16. 
Checking \cite[Table 4.5.1]{GLS}, we see that there is a subgroup $C$ (the $p'$-part of the centralizer of an involution of $^3\!D_4(q)$) of type $(A_1(q) \times A_1(q^3))/S$ such that $Z(A_1(q))=Z(A_1(q^3))=\langle m \rangle$ and $S:=\{(1,1),(m,m)\}$. Since the two $A_1$-factors are of simply connected type their Sylow 2-subgroups   - denoted by $Q$ and $Q'$ respectively - are generalized quaternion. Clearly, $m$ is contained in both of them and is moreover also contained in their commutator subgroups. Hence, $S$ is contained in $Q \times Q'$ and in $[Q,Q] \times [Q',Q']$. In particular $(Q \times Q')/S$ is a Sylow 2-subgroup of $C$ and is therefore contained in a Sylow 2-subgroup of $^3\!D_4(q)$. We have
\begin{align*}
((Q \times Q')/S) /([Q,Q] \times [Q',Q']/S) &\cong (Q \times Q') /([Q,Q] \times [Q',Q']) \\
&\cong Q/[Q,Q] \times Q'/[Q',Q'] \\
&\cong C_2 \times C_2 \times C_2 \times C_2,
\end{align*}
where the last isomorphism is a general property of generalized quaternion groups. Hence, $4 \leq s(B)$.
\end{proof} \text{} \\
\begin{center}
\textsc{\large{The block numbered $13$}} 
\end{center}
\noindent
We demonstrate the ideas for $G^F=E_{6,sc}(q)$. The arguments for $^2\!E_{6,sc}(q)$ are similar. Let $B=b_{G^F}(L, \lambda)$ be the block of $G^F$ numbered 13. In particular, $\ell=3$ and $e=1$ and $B$ corresponds to $s \in G^{*F}$ with $C_{G^*}(s)=A_5 \times A_1$.
To prove the assertions of Theorems \ref{Theorem A} and \ref{Theorem B} for $s$ and $B$ respectively, we will use block theory to shift from the simply connected group to their dual group, which is of adjoint type. Here, the problems with disconnected centralisers do not arise. We will proceed as follows:
\begin{enumerate}
\item Determine an upper bound on $\ell(B)$ via the adjoint groups.
\item Determine a lower bound on $s(B)$ via the simply-connected groups.
\end{enumerate}
This approach is supported by the following diagram (which follows from \cite[Proposition 24.21]{Malle-Testerman} for example).
\begin{center}
\begin{tikzcd}
G_{ad}^F & G_{sc}^F \arrow[twoheadrightarrow]{d} \\
\left[G_{ad}^F,G_{ad}^F\right] \arrow[hookrightarrow]{u} & G_{sc}^F/Z(G_{sc}^F) \arrow[leftrightarrow]{l}[swap]{\cong} 
\end{tikzcd}
\end{center} 
By the theory of dominating blocks (see e.g. \cite[Chapter 9]{Navarro}), $B$ dominates a unique block $\overline{B}$ of $G_{sc}^F/Z(G_{sc}^F)$ with a defect group $\overline{D}=D/Z(G_{sc}^F)$. The isomorphism  $\left[G_{ad}^F,G_{ad}^F\right] \cong G_{sc}^F/Z(G_{sc}^F)$ yields a block of $\left[G_{ad}^F,G_{ad}^F\right]$ isomorphic to $\overline{B}$ with the same defect group. We denote that block by $\overline{B}$ again. By the theory of covering blocks (see e.g. \cite[Chapter 9]{Navarro}), $\overline{B}$ is covered by a unique block $\tilde{B}$ with a defect group $\tilde{D}$ satisfying $\tilde{D} \cap \left[G_{ad}^F,G_{ad}^F\right]= \overline{D}$. We can say even more. Let $T(\overline{B})$ be the inertial group of $\overline{B}$ in $G_{ad}^F$. Since $G_{ad}^F/\left[G_{ad}^F,G_{ad}^F \right]$ is a group of order 3, there are only two options for $T(\overline{B})$. 
\begin{center}
(1) $T(\overline{B})=\left[G_{ad}^F,G_{ad}^F \right]$ \text{ }\text{ }\text{ }\text{ } or \text{ }\text{ }\text{ }\text{ } (2) $T(\overline{B})=G_{ad}^F$
\end{center}
If we are in case (1), then $\tilde{D}=\overline{D}$ and $l(\tilde{B})=l(\overline{B})$ by \cite[(9.14) Theorem]{Navarro}. In particular, $|\tilde{D}|=\frac{|D|}{3}$. If we are in case (2), then $d(\tilde{B})=d(\overline{B})+1$ by \cite[(9.17) Theorem]{Navarro} and $l(\overline{B}) \leq 3 \cdot l(\tilde{B})$ by Clifford theory and the fact that every irreducible Brauer character of $\left[G_{ad}^F,G_{ad}^F\right]$ is covered by an irreducible Brauer character of $G_{ad}^F$. In particular, $|\tilde{D}|=|D|$.

Since $l(\overline{B})=l(B)$, we have $l(B)=l(\tilde{B})$ in case (1) and $l(B) \leq 3 \cdot l(\tilde{B})$ in case (2). By Theorem \ref{E_6 adjoint type}, $c(\tilde{B}):=|\operatorname{Irr}(B) \cap \mathcal{E}(G^F,s)|$ is an upper bound for $l(\tilde{B})$. Hence, in case (1) it suffices to show
\begin{align*}
l(B) = l(\tilde{B}) \leq c(\tilde{B}) \leq 3^{s(B)},
\end{align*}
and in case (2) it suffices to show
\begin{align*}
l(B) \leq 3 \cdot l(\tilde{B}) &\leq 3 \cdot c(\tilde{B}) \leq \ell^{s(B)} \\
\text{or } c(\tilde{B}) &\leq 3^{s(B)-1}.
\end{align*}

\begin{remark} Let $B$ be an arbitrary block of $G_{sc}^F$ corresponding to a semisimple $\ell'$-element $s \neq 1$. Let $\pi$ denote the projection from $G_{sc}$ to $G_{ad}$. In general it is not known if $\tilde{B}$ corresponds to $\pi(s)$. So far, this has only been proved for unipotent $\ell$-blocks when $\ell$ is a good prime for $G$
(see \cite[Theorem 12]{Cabanes-Enguehardunipotent}). 
\end{remark}
\noindent
In any case, we are not able to immediately transition from $B$ to $\tilde{B}$ as we lack the necessary theory. However, we know that $|\tilde{D}|$ is either $|D|$ or $\frac{|D|}{3}$. Hence, if we denote set of blocks of $G_{ad}^F$ with defect groups of order $|D|$ or $\frac{|D|}{3}$ by $\mathcal{S}(B)$ then $\tilde{B} \in \mathcal{S}(B)$.

We can determine $\mathcal{S}(B)$ using $e$-Harish-Chandra theory. First note that similar arguments as in the proofs in \cite{Malle-Kessar} also work for the groups of adjoint type and are, in fact, much easier (since there are no disconnected centralisers). Furthermore, the results in \cite{Malle-Kessar} can easily be extended to non-quasi-isolated blocks. In particular, if an arbitrary block $\tilde{B}$ (with defect group $\tilde{D}$) of $G_{ad}^F$ corresponds to the $e$-cuspidal pair $(\tilde{L}, \tilde{\lambda})$ of $G_{ad}$ then $Z(\tilde{L})_\ell^F \subseteq \tilde{D}$ and the order of $\tilde{D}$ is known.

To prove the Malle--Robinson conjecture for $B$ it therefore suffices to prove
\begin{align*}
c(B') \leq 3^{s(B)-1} \text{ }\text{ }\text{ }\text{ }\forall B' \in \mathcal{S}(B).
\end{align*}
The sectional $3$-rank of $B$ is at least 6 since $Z(L)_3^F=\Phi_1^6 \subseteq D$. Moreover, $c(B')$ can be determined for every $B' \in \mathcal{S}(B)$.
It turns out that
\begin{align*}
c(B') \leq |\operatorname{Irr}(B) \cap \mathcal{E}(G^F,s)| < 3^{s(Z(L)_3^F)-1}=243 \leq 3^{s(B)-1} \qquad \forall B' \in \mathcal{S}(B),
\end{align*}
where $s$ is a quasi-isolated element of $G_{sc}^F$ associated to $B$. Hence the Malle--Robinson conjecture holds for $B$. 
Let $\overline{B}$ be the block of $G_{sc}^F/Z(G_{sc}^F)$ with defect group $\overline{D}$ dominated $B$. Then $\overline{D} \cong D/Z(G_{sc}^F)$. Since $Z(G_{sc}^F)=C_3$, $s(\overline{D})$ is either $s(D)$ or $s(D)-1$. So to show the conjecture for the block $\overline{B}$ of $G_{sc}^F/Z(G_{sc}^F)$ it certainly suffices to show that 
\begin{align}
c(B') \leq 3^{s(B)-2} \qquad \forall B' \in \mathcal{S}(B).
\end{align}
Since $3^{s(B)-2} \geq 3^{6-2}=81$ is greater than the size of every Lusztig series of $G_{ad}^F$, the conjecture holds for $\overline{B}$ by the same arguments as above. 


\noindent
As a corollary of this section we get the following result.

\begin{cor} \label{Theorem A and B for E6 and ell=3}
Let $\ell=3$, $e=1$ and $s \in G^{*F}$ be semisimple such that $C_{G^*}(s)=A_5 \times A_1$. Then 
$l_s \leq 3 \cdot |\mathcal{E}(G^F,s)|$. If $B$ is the block of $G^F$ numbered 13 or the block of $G^F/Z(G^F)$ dominated by that block then the Malle--Robinson conjecture holds for $B$.
\end{cor}

\begin{proof}
We demonstrate the proof for the case $G^F=E_6(q)$. Here, $\mathcal{E}_3(G^F,s)= \operatorname{Irr}(B)$ where $B$ is the block of $G^F$ numbered 13. By the above we therefore have $l_s=l(B) \leq 3 \cdot \mathcal{E}(G^F,s)$. 
\end{proof}

With this the assertions of Theorem \ref{Theorem A} and Theorem \ref{Theorem B} have been proved for the groups of type $E_6$.

\section{The quasi-isolated blocks of $E_7(q)$}
\noindent
Let $G$ be a simple, simply-connected algebraic group of type $E_7$ defined over $\mathbb{F}_q$ with Frobenius endomorphism $F:G \to G$. Since the center of $G$ is disconnected, we encounter the same intricacies we encountered for $E_6$. \\

\noindent
Let $\ell$ be a bad prime for $G$  not dividing $q$.
Let $1 \neq s \in G^{*F}$ be a semisimple, quasi-isolated $\ell'$-element and let $t \in C_{G^*}(s)_\ell^F$. Checking Table 1, we see that elements of order 6 are not isolated and elements of order greater than 6 are not quasi-isolated in $G^*$. 
\begin{lem} \label{Levis simply-connected E7} Let $L^* \subseteq G^*$ be a proper Levi subgroup of $G^*$.  Then $[L^*,L^*]$ is simply connected unless $L^*$ is of one of the following types: $D_6, A_5 \times A_1,A_3 \times  A_2 \times A_1,D_5 \times A_1,A_5,D_4 \times A_1,A_3 \times A_1^2,A_2 \times A_1^3, A_3 \times A_1,A_1^4,A_1^3$.
\end{lem}
\begin{proof}
This can be checked using Chevie \cite{Chevie}.
\end{proof}

\begin{prop}\label{generating set E7 l=3}
Let $\ell=3$ and let $s \in G^{*F}$ be a quasi-isolated semisimple $3'$-element. Then $\hat{\mathcal{E}}(G^F,s) \cup (\bigcup_t \hat{\mathcal{E}}(G^F,st)$ generates  $\mathbb{Q}\operatorname{IBr}(\mathcal{E}_3(G^F,s))$,  where $t$ runs over the $3$-elements of $C_{G^*}(s)^F$ such that 
\begin{enumerate}
\item $st$ is quasi-isolated; and
\item $C_{G^*}(st)^F \neq C_{G^*}^\circ(st)^F$ or $C_{G^*}^\circ(st)$ is not of type $A$.
\end{enumerate} 
In particular, $l_s \leq |\mathcal{E}(G^F,s)|+ \sum_{1 \neq t}|\mathcal{E}(G^F,st)|$.
\end{prop}

\begin{proof}
Similar to the proof of Proposition \ref{generating set quasi-isolated of order 3}. We use that the Levi subgroups of type $E_6$ have a simply connected derived subgroup.
\end{proof}

\begin{remark} \label{quasi-isolated of order 6 E7}
Let $z$ be a quasi-isolated element of order 6 in $G^*$. It can be shown (using Chevie for example) that $z=st$ where $s$ is quasi-isolated of order 2 with $[C_{G^*}^\circ(s),C_{G^*}^\circ(s)]=E_6$, and $t$ is quasi-isolated of order 3 with $C_{G^*}(t)=A_5 \times A_2$ (or vice-versa).
\end{remark}

\noindent

\begin{thm} \label{Conjecture E7}
Let $G$ be a simple, simply connected algebraic group of type $E_7$ defined over $\mathbb{F}_q$ with Frobenius endomorphism $F:G \to G$. Let $\ell \nmid q$ be a bad prime for $G$. 
Then the Malle--Robinson Conjecture holds for the quasi-isolated $\ell$-blocks of $G^F$ and $G^F/Z(G^F)$.
\end{thm}




\begin{proof}
By Ennola duality we can assume that $e=1$. Let $\ell=3$. Except for the blocks numbered 2, 8, 9, 10 and 11, it is, first of all, easy to determine $c(B)$ and secondly, $s(Z(L)_3^F)$ suffices to establish the Malle--Robinson conjecture where $(L, \lambda)$ is the $e$-cuspidal pair associated to the given block. For the block numbered 2, line 2b of \cite[Table 4]{Malle-Kessar} yields a sufficient lower bound on $s(B)$. To prove the conjecture for the blocks of type 8, 9, 10 or 11, we need to determine how the Lusztig series corresponding to the $\ell'$-elements satisfying conditions (1) and (2) of Proposition \ref{generating set E7 l=3} decompose into $3$-blocks. Recall that
\small
\begin{align*}
|G^F|=q^{63} \Phi_1(q)^7 \Phi_2(q)^7 \Phi_3(q)^3 \Phi_4(q)^2 \Phi_5(q) \Phi_6(q)^3 \Phi_7(q) \Phi_8(q) \Phi_9(q) \Phi_{10}(q) \Phi_{12}(q) \Phi_{14}(q) \Phi_{18}(q)
\end{align*}
\normalsize
By the assumption on $e$, the only $\Phi_i(q)$, appearing in the expression above that are divisible by 3 are $\Phi_1, \Phi_3$ and $\Phi_9$. While $\Phi_1(q)$ can be divisible by higher powers of 3 (depending on $q$), $\Phi_3(q)$ and $\Phi_9(q)$ are only divisible by 3. Hence,
\begin{align*}
|G^F|_3=|\Phi_1(q)|_3^7 \text{ }|\Phi_3(q)|_3^3 \text{ }|\Phi_9(q)|_3=3^4 \text{ }|\Phi_1(q)|_3^7.
\end{align*}
Let $B=b_{G^F}(L,\lambda)$ be a block of type 8, 9, 10 or 11 and let $D$ be a defect group of $B$. By \cite[Theorem 1.2]{Malle-Kessar} we know that $D$ is a Sylow 3-subgroup of an extension of $Z(L)_3^F$ by $W_G(L, \lambda)$. Hence, $|D|=|Z(L)_3^F||W_G(L,\lambda)|_3$. By the definition of the defect of $B$ (see \cite[Definition (3.15)]{Navarro}) we have
\begin{align}
|G^F|_3/|D|= \text{min}\{\chi(1)_3 \, | \, \chi \in \operatorname{Irr}(B)\}.
\end{align}\label{defect minimum}
\noindent
We get the following table.

\begin{center} $
\begin{array}{|c|r|r|} \hline
B	&	|D|	&	|G^F|_3/|D| \\ \hline
8	&	3^4 \text{ } |\Phi_1(q)|_3^7	&	1 \\
9	&	3 \text{ } |\Phi_1(q)|_3^3 	&	3^3 \text{ } |\Phi_1(q)|_3^4 \\ \hline
10	&	3^2 \text{ } |\Phi_1(q)|_3^4 &	3^2 \text{ } |\Phi_1(q)|_3^3 \\
11	&	|\Phi_1(q)|_3	&	3^4 \text{ } |\Phi_1(q)|_3^6 \\ \hline
\end{array} $
\end{center}

We start  with the blocks numbered 8 and 9. Let $s \in G^{*F}$ be the semisimple element corresponding to the blocks numbered 8 and 9. By Proposition \ref{generating set E7 l=3}, $\hat{\mathcal{E}}(G^F,s) \cup \hat{\mathcal{E}}(G^F,st)$, where $t \in C_{G^*}(s)_3^F$ such that $C_{G^*}(st)^F=\Phi_1 A_2(q)^3.2$, generates $\mathbb{Q}\mathcal{E}_3(G^F,s)$. We claim that the series $\mathcal{E}(G^F,st)$ is contained in the block numbered 8. 
Let $\Psi_{st}$ denote the Jordan decomposition associated with $st$ (see \cite[Corollary 15.14]{Cabanes-Enguehard}). Let $\chi \in \mathcal{E}(G^F,st)$. By \cite[Remark 13.24]{Digne-Michel} we have
\begin{align*}
\chi(1)_3= \frac{|G^F|_3}{|C_{G^*}(st)^F|_3}\Psi_{st}(\chi)(1)_3.
\end{align*}
The right side of this equation can easily be computed and we observe that $\chi(1)_3 <  3^3 |\Phi_1(q)|_3^4$ for every $\chi \in \mathcal{E}(G^F,st)$.
So it follows from (2) that $\mathcal{E}(G^F,st)$ is fully contained in the block numbered 8. We argue similarly for the blocks of type 10 and 11. It can be shown that the Lusztig series corresponding to the quasi-isolated element of order 6 (appearing in the generating set for the union of the blocks of type 10 and 11) is contained in the blocks of type 10. 

For the quasi-isolated blocks of $G^F/Z(G^F)$ we use the arguments of the proof of Theorem \ref{Conjecture E6}.

Let $\ell=2$ now. Since we assumed that $e=1$, a quasi-isolated 2-block of $G^F$ or $G^F/Z(G^F)$ is either of type $1_u$. $3_u$, 1, 2, respectively dominated by one of them. For the blocks numbered $1_u$ and $2_u$, $s(Z(L)_2^F)$ suffices to establish the conjecture. For the blocks numbered 1 and 2 we will argue as we did for the block numbered $13$ of $E_6(q)$. The assertion then follows from Corollary \ref{Proof Theorem A and B E7 l=2} below.
\end{proof} \text{} \\

\begin{center}
\textsc{\large{The blocks numbered $1$ and $2$}} 
\end{center}
\noindent
In this section we finish the proof of Theorem \ref{Conjecture E7} (therefore finishing the proof of the assertion of Theorem \ref{Theorem B} for $E_7(q)$) and the proof of the assertion of Theorem \ref{Theorem A} for $E_7(q)$. As before we can assume that $e=1$.

\begin{thm} \label{E_7 adjoint type}  Let $G$ be a simple algebraic group of adjoint type $E_7$ defined over $\mathbb{F}_q$ with Frobenius endomorphism $F: G \to G$. Let $\ell \nmid q$ be a bad prime for $G$ and let $s \in G^{*F}$ be a semisimple $\ell'$-element. Then $\hat{\mathcal{E}}(G^F,s)$ generates $\mathbb{Q}\operatorname{IBr}(\mathcal{E}_\ell(G^F,s))$, unless possibly if $\ell=2$ and 
\begin{enumerate}
\item $s=1$, or
\item $C_{G^*}(s)=D_6$. 
\end{enumerate}
In particular, $l_s \leq |\mathcal{E}(G^F,s)|$ unless $s$ satisfies (1) or (2).
\end{thm}

\begin{proof}
Suppose that $\ell=3$. Let $1 \neq s \in G^{*F}$ be a semisimple $3'$-element and let $t \in C_{G^*}(s)_3^F$. Then, either $C_{G^*}(st)=A_5 \times A_2$ or $st$ is not quasi-isolated in $G^*$. In the first case, we conclude as we did before for connected centralisers of type $A$. Hence, suppose that $st$ is not quasi-isolated in $G^*$. Let $L^*$ be the minimal Levi subgroup containing $C_{G^*}(st)$. In particular, $st$ is quasi-isolated in $L^*$. If $L^*$ is of type $E_6$ then either $t \in Z(L^*)$ or $C_{G^*}(st)$ is of type $A$. If $L^*$ is of classical type, then 3 is a good prime for $L^*$. So we are done by Theorem \ref{restriction}. 

Let $\ell=2$. Let $s \in G^{*F}$ be a semisimple $2'$-element and let $t \in C_{G^*}(s)_2^F$. If conditions (1) and (2) are not satisfied then either $t \in Z(L^*)$, where $L^*$ is the minimal Levi subgroup of $G^*$ containing $C_{G^*}(st)$, or $C_{G^*}(st)$ is of type $A$. However, if (1) or (2) are satisfied there exist $t \in C_{G^*}(s)_2^F$ such that $C_{G^*}(st)=C_{G^*}(t)=D_4 \times A_1^2$. In this case none of our methods can be applied.
\end{proof}

Let $\ell=2$. For a 2-block $B$ of $G_{sc}^F$ with defect group $D$ we denote the set of blocks of $G_{ad}^F$ with defect groups of order $|D|$ or $\frac{|D|}{2}$ by $\mathcal{S}(B)$. For the blocks $1$ and $2$ we will use the same approach that we used for the block numbered 13 in the last section. Let $s \in G_{sc}^{*F}$ be a quasi-isolated element corresponding to the block numbered 1 (respectively 2). By \cite{Malle-Kessar}, $\mathcal{E}_2(G^F,s)$ consists of only one block. Let $B$ be numbered 1 (respectively 2) and let $D$ be a defect group of $B$. Note that the blocks corresponding to the two exceptions in Theorem \ref{E_7 adjoint type} do not lie in $\mathcal{S}(B)$. As before we define $c(B')=|\operatorname{Irr}(B') \cap \mathcal{E}(G_{ad}^F,s')|$, where $s' \in G_{ad}^{*F}$ is associated to the block $B'$ of $G_{ad}^F$. It turns out that
\begin{align}
c(B')  \leq |\mathcal{E}(G^F,s)| < 2^{s(Z(L)_2^F)-1}=64 \leq 2^{s(B)-1} \qquad \forall B' \in \mathcal{S(B)}.
\end{align}

Let $\overline{B}$ be the block of $G_{sc}^F/Z(G_{sc}^F)$ dominated by $B$. For the block numbered 1 we have $s(Z(L)_2^F)=s(Z(L)_2^F/Z(E_{7,sc}(q))$ since $4 \mid (q-1)$ for $e=1$ (see \cite[Table 4]{Malle-Kessar}). Similarly we argue for the block numbered 2 using line 2b. Hence, it suffices to show 
\begin{align*}
c(B') \leq  |\mathcal{E}(G^F,s)|  < 2^{s(Z(L)_2^F)-1}=64  \qquad \forall B' \in \mathcal{S}. 
\end{align*}
However, this was established above when we proved the conjecture for the block $B$.


\noindent
As a corollary of these arguments we have the following.

\begin{cor}\label{Proof Theorem A and B E7 l=2}
Let $\ell=2$ and let $s \in G^{*F}$ be a quasi-isolated semisimple $2'$-element. Then 
$l_s \leq 2 \cdot |\mathcal{E}(G^F,s)|$. Moreover, if $B$ is a quasi-isolated block of $G^F$ or $G^F/Z(G^F)$ then the Malle--Robinson conjecture holds for $B$.
\end{cor}

\begin{proof}
Similar to the proof of Corollary \ref{Theorem A and B for E6 and ell=3}.
\end{proof}

\section{The quasi-isolated blocks of $E_8(q)$}
\noindent
Let $G$ be a simple, simply connected algebraic group of type $E_8$ defined over $\mathbb{F}_q$ with Frobenius endomorphism $F: G \to G$. Recall that simple algebraic groups of type $E_8$ are both simply connected and adjoint.  We will therefore omit any specification of the isogeny type as we did in Section 3.

\begin{thm}\label{generating set E8}
Let $\ell \nmid q$ be a bad prime for $G$ and let $s \in G^{*F}$ be a semisimple quasi-isolated $\ell'$-element. Then $\hat{\mathcal{E}}(G^F,s) \cup (\bigcup_t \hat{\mathcal{E}}(G^F,st)$ generates  $\mathbb{Q}\operatorname{IBr}(\mathcal{E}_\ell(G^F,s))$,  where $t$ runs over the $\ell$-elements of $C_{G^*}(s)^F$ such that 
\begin{enumerate}
\item $st$ is quasi-isolated; and
\item  $C_{G^*}(st)$ is not of type $A$.
\end{enumerate} 
In particular, $l_s \leq |\mathcal{E}(G^F,s)|+ \sum_{1 \neq t}|\mathcal{E}(G^F,st)|$.
\end{thm}

\begin{proof}
Similar to the proofs of Theorem \ref{generating set F4} and Propositions \ref{generating set quasi-isolated of order 3} and \ref{generating set E7 l=3}.
%
\end{proof}

%

\begin{thm}\label{Conjecture E_8}
Let $G$ be a simple, simply connected algebraic group of type $E_8$ defined over $\mathbb{F}_q$ with Frobenius endomorphism $F: G \to G$.
Let $\ell \nmid q$ be a bad prime for $G$. Then the Malle--Robinson conjecture holds for  the quasi-isolated $\ell$-blocks of $G^F$ unless, possibly, if $B$ is the block numbered $3$ or $8$ in the table in \cite[page 358]{Enguehard}.
\end{thm}

\begin{proof}
First, suppose that $\ell=2$. Let $B=b_{G^F}(L, \lambda)$ be a quasi-isolated 2-block of $G^F$. Except for the blocks of type $3_u$, $8_u$, 2, 8 and 9, $s(Z(L)_2^F)$ suffices to establish the conjectured upper bound.
Let $B=b_{G^F}(L, \lambda)$ be numbered 2, 8 or 9. Recall that we have a normal series
\begin{align*}
Z(L)_l^F \unlhd P:=C_D(Z(L)_2^F) \unlhd D,
\end{align*}
where $D$ is a defect group of $B$. Furthermore, by \cite[Proposition 2.1 and 2.7]{Malle-Kessar}, $P$ is a defect group of the block of $L^F$ containing $\lambda$.  Now, in all cases (2, 8 and 9) $C_{L^*}(s)$ is a maximal torus of $L^*$. Let $M \subset G$ be an $F$-stable torus dual to $C_{L^*}(s)$. There is a Morita equivalence 
\begin{align*}
\mathcal{O}L^Fe_s^{L^F} \sim \mathcal{O}M^Fe_1^{M^F},
\end{align*}
(see Theorem \ref{BDR}) with a bijection between the blocks on both sides preserving defect groups. In particular, $s(D)=s(D')$ where $D'$ is a defect group corresponding to $D$ by this bijection.
Since $M$ is a torus, there is only one block on the right side of the equivalence, namely the principal block of $M^F$. Every defect group of that block is a Sylow 2-subgroup of $M^F$. Now, the structure of $M^F$ can be read off from \cite[Table 5]{Malle-Kessar}. Hence, we can determine $s(D)$ and observe that $c(B) \leq 2^{s(D)}$.
If $B=b_{G^F}(L, \lambda)$ is numbered $3_u$ or $8_u$ then we have $c(B)=|\operatorname{Irr}(B) \cap \mathcal{E}(G^F,1)|=6$ but $s(Z(L)_2^F)=2$. So we either need a better bound for $l(B)$ or a better understanding of the defect groups of $B$ to establish the conjectured upper bound (or the conjecture is false). So far both are missing.

Now, suppose that $\ell=3$ or $5$. Let $B=b_{G^F}(L, \lambda)$ be a quasi-isolated $\ell$-block of $G^F$. In this case, $s(Z(L)_\ell^F)$ suffices to establish the conjectured upper bound.
\end{proof}
\noindent

\section{Proofs of the main statements}
\noindent
The proofs of Theorems \ref{Theorem A} and \ref{Theorem B} are given by combining the results of the previous sections.

\begin{proof}[Proof of Theorem \ref{Theorem A}]
The assertion of Theorem \ref{Theorem A} follows from Theorem \ref{generating set F4}, Proposition \ref{generating set quasi-isolated of order 3}, Corollary \ref{Theorem A and B for E6 and ell=3}, Proposition \ref{generating set E7 l=3}, Corollary \ref{Proof Theorem A and B E7 l=2} and Theorem \ref{generating set E8}.
\end{proof}

\begin{proof}[Proof of Theorem \ref{Theorem B}]
If $\ell$ is good for $G$ the assertion follows from \cite[Theorem B]{Ruwen} and if $\ell$ is bad for $G$ the assertion follows from Theorems \ref{Conjecture F4}, \ref{Conjecture E6}, \ref{Conjecture E7} and \ref{Conjecture E_8}.
\end{proof}
\noindent
Before we prove the Corollary to Theorem \ref{Theorem B} we introduce the object in question. Let $H$ be a finite group and let $B$ be an $\ell$-block of $H$. Then $(H,B)$ (or just $B$, if $H$ is understood) is called a \textbf{minimal counterexample} to the Malle--Robinson conjecture if 
\begin{enumerate}
\item the conjecture does not hold for $B$, and
\item the conjecture holds for all $\ell$-blocks $B'$ of groups $K$ with $|K/Z(K)|$ strictly smaller than $|H/Z(H)|$ having defect groups isomorphic to those of $B$.
\end{enumerate} 


\begin{proof}[Proof of Corollary]
Suppose that $(H,B)$ is a minimal counterexample to the Malle--Robinson conjecture. Let $D$ be a defect group of $B$. By \cite[Proposition 6.4]{Malle-Robinson}, $H$ is not an exceptional covering group of a finite group of exceptional Lie type. By \cite[Proposition 6.5]{Malle-Robinson}, $H$ is not of Lie type $^2\!B_2$, $^2\!G_2$, $G_2$, $^3\!D_4$ or $^2\!F_4$. Hence, $H=G^F/Z$, where $G$ is a simple, simply connected group of exceptional type ($F_4,E_6,E_7$ or $E_8$), $F:G \to F$ is a Frobenius endomorphism and $Z \subseteq Z(G^F)$ is a central subgroup. By \cite[Proposition 6.1]{Malle-Robinson}, $\ell$ does not divide $q$. Let $B'$ be the unique block of $G^F$ that dominates $B$ and let $D'$ be a defect group of $B'$. In particular, $l(B)=l(B')$ and $s(D)=s(D')$. By \cite[Theorem 7.7]{BDR}, $B'$ is Morita equivalent to an $\ell$-block $b$ of a subgroup $N$ of $G^F$ and their defect groups are isomorphic. In particular, $l(B')= l(b)$ and $s(B')=s(b)$. If $s$ is not quasi-isolated, then $N$ is a proper subgroup. By the minimality of $(H,B)$, $B$ is therefore a quasi-isolated block of $H$. By \cite[Theorem B]{Ruwen}, $\ell$ is bad for $G$. By Theorem \ref{Theorem B}, if a minimal counterexample exists it would be $(E_8(q),B_1)$, where $B_1$ is one of the 2-blocks numbered $3_u$ or $8_u$.
\end{proof}

\bibliography{Literatur}
\addcontentsline{toc}{section}{References}
\bibliographystyle{plain}

\end{document}